\documentclass[12pt]{amsart}
\usepackage{amssymb}
\usepackage{amsfonts}
\usepackage{amssymb,latexsym}
\usepackage{enumerate}
\usepackage{mathrsfs}
\usepackage{hyperref}
 \usepackage{color}\makeatletter
\@namedef{subjclassname@2010}{%
  \textup{2010} Mathematics Subject Classification}
\makeatother

  [2002/01/19 v2.2g %
    AMS font definitions%
  ]
\DeclareFontFamily{U}{euf}{}
\DeclareFontShape{U}{euf}{m}{n}{%
  <5><6><7><8><9>gen*eufm%
  <10><10.95><12><14.4><17.28><20.74><24.88>eufm10%
  }{}
\DeclareFontShape{U}{euf}{b}{n}{%
  <5><6><7><8><9>gen*eufb%
  <10><10.95><12><14.4><17.28><20.74><24.88>eufb10%
  }{}

  [2002/01/19 v2.2g %
    AMS font definitions%
  ]
\DeclareFontFamily{U}{msb}{}
\DeclareFontShape{U}{msb}{m}{n}{%
  <5><6><7><8><9>gen*msbm%
  <10><10.95><12><14.4><17.28><20.74><24.88>msbm10%
  }{}

  [2002/01/19 v2.2g %
    AMS font definitions%
  ]
\DeclareFontFamily{U}{msa}{}
\DeclareFontShape{U}{msa}{m}{n}{%
  <5><6><7><8><9>gen*msam%
  <10><10.95><12><14.4><17.28><20.74><24.88>msam10%
  }{}

\newtheorem{theorem}{Theorem}[section]
\newtheorem{lemma}[theorem]{Lemma}

\newtheorem{corollary}[theorem]{Corollary}

\theoremstyle{definition}

\newtheorem{remark}[theorem]{Remark}

\numberwithin{equation}{section} 

\def\t{\widetilde}
\def\ga{\gamma}

\begin{document}

\title[Gamma function and alternating Hurwitz zeta function]
{On gamma functions with respect to the alternating Hurwitz zeta functions}
\author{Wanyi Wang}
\address{Department of Mathematics, South China University of Technology, Guangzhou 510640, China}
\email{mawanyi@mail.scut.edu.cn}

\author{Su Hu}
\address{Department of Mathematics, South China University of Technology, Guangzhou 510640, China}
\email{mahusu@scut.edu.cn}

\author{Min-Soo Kim}
\address{Department of Mathematics Education, Kyungnam University, Changwon, Gyeongnam 51767, Republic of Korea}
\email{mskim@kyungnam.ac.kr}

\begin{abstract}
In 2021, Hu and Kim \cite{HK2022} defined a new type of gamma function $\widetilde{\Gamma}(x)$  from the alternating Hurwitz zeta function $\zeta_{E}(z,x)$, and obtained some of its properties.
In this paper, we shall further investigate the function $\widetilde{\Gamma}(x)$, that is, we obtain several properties in analogy to the classical Gamma function $\Gamma(x)$, including the integral representation, the limit representation, the recursive formula, the special values, the log-convexity, the duplication and distribution formulas, and the reflection equation. Furthermore, we also prove a Lerch-type formula, which shows that the derivative of $\zeta_{E}(z,x)$ can be representative by $\widetilde\Gamma(x)$.  \end{abstract}

\subjclass[2010]{11M35, 11M06, 33B15}
\keywords{Alternating Hurwitz zeta function, Gamma function, Digamma function, Hypergeometric function}

\maketitle 
\section{Introduction}
The gamma function $\Gamma(x)$ is defined by Euler from the integral representation
\begin{equation}\label{1.2}
\Gamma(x)=\int_{0}^{\infty}t^{x-1}e^{-t}dt
\end{equation}
for Re$(x)>0$. It is an extension of the factorial function $\Pi(n)=n!$ from the integral variable to the real and complex variables,
which has many applications in various branches of sciences, such as number theory, physics, statistics, and so on.
According to Lerch, Euler's gamma function $\Gamma(x)$ can also be defined by the derivative of the Hurwitz zeta function
\begin{equation}\label{Hu}
\zeta(z,x)=\sum_{n=0}^{\infty}\frac{1}{(n+x)^{z}}
\end{equation} 
at $z=0$, that is, we have 
\begin{equation}\label{Lerch-1}
\log\Gamma(x)=\zeta'(0,x)-\zeta'(0,1)=\zeta'(0,x)-\zeta'(0)
\end{equation}
(see \cite[Definition 9.6.13(1)]{Cohen}).
Here $$\zeta(z)=\zeta(z,1)=\sum_{n=1}^{\infty}\frac{1}{n^{z}}$$
is the Riemann zeta function. In addition, $\Gamma(x)$ has the following 
well-known Weierstrass--Hadamard product
\begin{equation}\label{Hadamard}
\Gamma(x)=\frac{1}{x}e^{-\gamma x} \prod_{k=1}^{\infty}\left(e^{\frac{x}{k}}\left(1+\frac{x}{k}\right)^{-1}\right),
\end{equation}
where $\gamma=0.5772156649\cdots$ is the Euler constant.
Furthermore, for Re$(x)>0$, let
\begin{equation}\label{digamma} \psi(x):=\frac{d}{dx}\log\Gamma(x)\end{equation}
be the digamma function (see \cite[Definition 9.6.13(2)]{Cohen} and \cite[p. 32]{FS}).
And if setting  \begin{equation}\psi^{(n)}(x):=\left(\frac{d}{dx}\right)^n\psi(x),
	\quad n\in \mathbb N_0=\mathbb{N}\cup\{0\},\end{equation}
then we have the following formula which represents the special values of $\zeta(z,x)$:
\begin{equation}\label{1.1} \psi^{(n)}(x)=(-1)^{n+1}n!\zeta(n+1,x), \quad n\in\mathbb N
\end{equation}
(see \cite[(1.11)]{CS} and \cite[Proposition 9.6.41]{Cohen}).

For Re$(z)>0$ and  $x\neq0,-1,-2,\ldots,$
let \begin{equation}\label{E-zeta-def}
\zeta_E(z,x)=\sum_{n=0}^\infty\frac{(-1)^n}{(n+x)^{z}}
\end{equation}
be the alternating form of the Hurwitz zeta function (also known as the Hurwitz-type Euler zeta function).
In particular, setting $x=1$ we obtain the alternating zeta function (also known as Dirichlet's eta function),
\begin{equation}\label{A-zeta}
\zeta_E(z)=\zeta_E(z,1)=\sum_{n=1}^\infty\frac{(-1)^{n+1}}{n^{z}}=\eta(z).
\end{equation}
And $\zeta(z)$ and $\zeta_{E}(z)$ satisfy 
\begin{equation}
\zeta_E(z)=\left(1-\frac1{2^{z-1}}\right)\zeta(z).
\end{equation}
It is noted that in the papers \cite{Chandel} and \cite{Srivastava}, using a relation  
\begin{equation}
\lim_{n\to\infty}\Gamma(z+n+1)=\lim_{n\to\infty}\Gamma(n+1)n^{z},
\end{equation}
the authors obtained several convergence conditions for various families of Hurwitz-Lerch  type zeta functions for all $z$ belongs to the complex field
(see, e.g., \cite[(1.8)]{Chandel}). 

During the recent years, the Fourier expansion, power series and asymptotic expansions, 
 integral representations, special values, and convexity properties  of $\zeta_{E}(z,x)$ have been systematically studied (see \cite{Cvijovic, HKK, HK2019, HK2022, HK2024}). In algebraic number theory, 
 it is found that $\zeta_{E}(z,x)$ can be used to represent a partial zeta function of cyclotomic fields in one version of Stark's conjectures (see \cite[p. 4249, (6.13)]{HK-G}).
In addition, the alternating zeta function $\zeta_E(z)$  is  a particular case of Witten's zeta functions in mathematical physics (see \cite[p. 248, (3.14)]{Min}).
And for a form containing in a handbook of mathematical functions by Abramowitz and Stegun, the left hand side gives the special values of the Riemann zeta functions $\zeta(z)$ at positive integers, 
and the right hand side gives the special values of $\zeta_{E}(z)$ at the corresponding points (see \cite[p. 811]{AS}).
  
In analogy to the classical situation (see \cite[(1.3)]{HK2022}), in 2021 Hu and Kim \cite{HK2022} defined the corresponding generalized  Stieltjes constant $\t\gamma_k(x)$ and the Euler constant $\t\gamma_{0}$
from the Taylor expansion of $\zeta_{E}(z,x)$ at $z=1$. That is, they designated a modified Stieltjes constants $\t\ga_k(x)$ by
the Taylor expansion of $\zeta_E(z,x)$ at $z=1$,
\begin{equation}\label{l-s-con}
\zeta_E(z,x)=\sum_{k=0}^\infty\frac{(-1)^k\t\ga_k(x)}{k!}(z-1)^k
\end{equation}
(see \cite[(1.18)]{HK2022}).
From the above expansion, we see that
\begin{equation}\label{gamma0(x)} 
\t\ga_0(x)=\zeta_E(1,x).
\end{equation}
Let \begin{equation}\label{gammak}\t\gamma_k:=\t\gamma_k(1)\end{equation}
 for $k\in\mathbb N_0,$
and by (\ref{A-zeta}), (\ref{gamma0(x)}) and (\ref{gammak}) we get 
\begin{equation}\label{1-35}
\widetilde{\gamma}_{0}=\widetilde{\gamma}_{0}(1)=\zeta_{E}(1)=\sum_{n=1}^{\infty}\frac{(-1)^{n+1}}{n}=\log 2
\end{equation}
(also see \cite[(1.17) and (1.20)]{HK2022}).
Then they defined the corresponding digamma function $\t\psi(x)$ by 
\begin{equation}\label{psi-def}
\t\psi(x):=-\t\ga_0(x)=-\zeta_E(1,x).
\end{equation}
(see \cite[(1.22)]{HK2022}),
which is essentially Nilsen's beta function $\beta(x)$ (see (\ref{1.13})).
And  the corresponding gamma function $\t\Gamma(x)$ is defined from the differential equation (see \cite[p. 5]{HK2022})
\begin{equation}\label{1.7}
\t\psi(x)=\frac{\t\Gamma'(x)}{\t\Gamma(x)}=\frac{d}{dx}\log\t\Gamma(x),
\end{equation}
which is in analogy to the classical formula (\ref{digamma}). Furthermore, they proved an analogue of the Weierstrass--Hadamard product
\begin{equation}\label{1.6}
\t\Gamma(x)=\frac1x e^{\t\gamma_0 x}\prod_{k=1}^\infty\left(e^{-\frac xk}\left(1+\frac xk\right)\right)^{(-1)^{k+1}}
\end{equation}
(see \cite[Theorem 3.12]{HK2022}). If setting \begin{equation}\label{poly-ga-def}
\t\psi^{(n)}(x):=\left(\frac{d}{dx}\right)^n\t\psi(x),\quad n\in\mathbb N_0,
\end{equation}
then as in the classical situation (\ref{1.1}), $\t\psi^{(n)}(x)$ also represents the special values of $\zeta_{E}(z,x)$:
\begin{equation}\label{ga-poly}
\t\psi^{(n)}(x)=(-1)^{n+1}n!\zeta_E(n+1,x), \quad n\in\mathbb N_0
\end{equation}
(see \cite[p. 957, 8.374]{GR} and \cite[(1.25)]{HK2022}). Recently, several authors have investigated the Stieltjes constants, gamma and digamma functions in relation to other zeta functions, such as hyperharmonic zeta and eta functions, as well as the Hurwitz-Lerch zeta function, see, e.g., \cite{Can2023, Can2024, CS, Choi, Cicimen, Can2023-2, Nakamura}.

It is known that the classical gamma function $\Gamma(x)$ has many interesting properties, such as the integral representation, the limit representation, the recursive formula, the log-convexity, the duplication and distribution formulas, 
and the reflection equation. These properties lead it wide applications in analysis, number theory, mathematical physics and probability.
In this paper, we shall show that these properties all have their analogues for $\widetilde{\Gamma}(x)$  (see Section 2). By comparing the integral representations of $\Gamma(x)$ and $\widetilde{\Gamma}(x)$,
 we see that for $\mathrm{Re}(x)>0$, $\Gamma(x)$ is the Mellin transform of the function $e^{-t}$, 
and $\widetilde{\Gamma}(x)$ is the Laplace transform of the function $(1-e^{-2t})^{-\frac{1}{2}}$ (see Corollary \ref{Cor 2.5+}). 
And from the integral representation of $\widetilde{\Gamma}(x)$ (Theorem \ref{Theorem 1.10}), we will prove that, as many elementary functions, 
$\widetilde{\Gamma}(x)$ can also be expressed in terms of the Gauss hypergeometric function $_2F_{1}(a,b;c;z)$:
\begin{equation}
\widetilde{\Gamma}(x)=\frac{_2F_{1}\left(\frac{1}{2},\frac{x}{2}; \frac{x}{2}+1;1\right)}{x}.
\end{equation}
for {\rm Re}$(x)>0$
 (see Corollary \ref{hyper}).

Then in Section 3, we will go on to investigate the properties of  the corresponding digamma function $\t\psi(x)$, 
 including its recursive formula and the reflection equation. In Section 4, we will study the relation between $\t\Gamma(x)$ and $\zeta_{E}(z,x)$. In concrete, we show that Lerch's formula (\ref{Lerch-1}) is also established in our situation, 
that is, we have 
\begin{equation}\label{Lerch-2-intro}
\log\t\Gamma(x)=\zeta_{E}'(0,x)+\zeta_{E}'(0).
\end{equation}

\section{The properties of $\t\Gamma(x)$}
In this section, we will investigate the properties of $\t\Gamma(x)$. First, we obtain its special value at 1.
\begin{lemma}\label{Lemma 3.1}
We have \begin{equation}\label{special} \t\Gamma(1)=\frac{\pi}{2}.\end{equation}
\end{lemma}
\begin{remark}
Recall the following special values for the classical gamma function $\Gamma(x)$:
$$\Gamma(1)=1\quad\textrm{and}\quad\Gamma\left(\frac1{2}\right)=\sqrt{\pi}.$$
\end{remark}
\begin{proof}[Proof of Lemma \ref{Lemma 3.1}]
We have the expansion 
\begin{align*}
\log\widetilde{\Gamma}(1)=\widetilde{\gamma}_{0}+\sum_{k=1}^{\infty}(-1)^{k}\left(\frac{1}{k}-\log\left(1+\frac{1}{k}\right)\right)
\end{align*}
(see \cite[(4.31)]{HK2022}).
Since 
\begin{equation}\label{gamma0}
\widetilde{\gamma}_{0}=\sum_{k=1}^{\infty}\frac{(-1)^{k+1}}{k}
\end{equation}
(see (\ref{1-35})),
by Wallis' formula we see that
\begin{align*}
\log\widetilde{\Gamma}(1)&=\sum_{k=1}^{\infty}(-1)^{k+1}\left(\log(k+1)-\log(k)\right)\\
                           &=\lim_{n \to \infty}\log\left(\frac{(2n!!)^{2}}{(2n-1)!!(2n+1)!!}\right)\\
                           &=\log\frac{\pi}{2},
\end{align*}
which is equivalent to (\ref{special}).
\end{proof}
Now we have the following integral representation of $\widetilde{\Gamma}(x)$.
\begin{theorem}[Integral representation]\label{Theorem 1.10}
For {\rm Re}$(x)>0$, we have  
\begin{equation}\label{Th 2.3}
\widetilde{\Gamma}(x)=\frac{1}{2}\mathrm{B}\left(\frac{x}{2},\frac{1}{2}\right)
=\frac{1}{2}\int_{0}^{1}\frac{t^{\frac{x}{2}-1}}{\sqrt{1-t}}dt,
\end{equation}
where $B(x,y)$ is the Beta function (also known as the first Eulerian integral).
\end{theorem}
\begin{remark}
Recall the integral representation for the classical gamma function $\Gamma(x)$:
$$\Gamma(x)=\int_{0}^{\infty}e^{-t}t^{x-1}dt=\int_{0}^{1}\left(\log\frac{1}{t}\right)^{x-1}dt.$$
\end{remark}
\begin{proof}[Proof of Theorem \ref{Theorem 1.10}]
As pointed out  in 
\cite[p. 5]{HK2022},  $\t\psi(x)$ is equivalent to Nilsen's $\beta$-function $\beta(x)$, that is, we have 
\begin{equation}\label{1.13}
\beta(x)=\frac12\left(\psi\left(\frac{x+1}{2}\right)-\psi\left(\frac{x}{2}\right)\right)=-\t\psi(x)
\end{equation}
(comparing \cite[(1.3)]{BMM}, \cite[(2.5)]{HK2022} and \cite[(1.4)]{Nanto2017}).
Furthermore, as shown by Nantomah, Nilsen's $\beta$-function $\beta(x)$ can be
represented by the derivative of the classical Beta function $\textrm{B}(x,y)$:
\begin{equation}\label{1.4}
\beta(x)=-\frac{d}{dx}\log \textrm{B}\left(\frac{x}{2},\frac{1}{2}\right)
\end{equation}
for $x>0$ (see \cite[Proposition 1.1]{Nanto2017}).
Then comparing (\ref{1.7}), (\ref{1.13}) and (\ref{1.4}), we get
$$\frac{d}{dx}\log\t\Gamma(x)=\t\psi(x)=-\beta(x)=\frac{d}{dx}\log \textrm{B}\left(\frac{x}{2},\frac{1}{2}\right),$$
thus
\begin{equation}\label{integralA}
\widetilde{\Gamma}(x)=C\mathrm{B}\left(\frac{x}{2},\frac{1}{2}\right)
=C\int_{0}^{1}t^{\frac{x}{2}-1}(1-t)^{-\frac{1}{2}}dt=C\int_{0}^{1}\frac{t^{\frac{x}{2}-1}}{\sqrt{1-t}}dt,
\end{equation}
where $C$ is a constant.
Then letting $x=1$ in the above equation and by noticing (\ref{special}), we have 
$$\frac{\pi}{2}=\widetilde{\Gamma}(1)=C\mathrm{B}\left(\frac{1}{2},\frac{1}{2}\right)=C\pi.$$
From which, we determine $C=\frac{1}{2}$.
Then substituting into (\ref{integralA}), we get the desired result.
\end{proof}
From the above theorem and the relation between $\Gamma(x)$ and $\textrm{B}(x,y)$ (see \cite[Proposition 9.6.39]{Cohen}), we immediately get the following result.
\begin{corollary}\label{Corollary 3.3}
For {\rm Re}$(x)>0$, we have
\begin{equation}\label{Cor 2.5}
\widetilde{\Gamma}(x)=\frac{1}{2}\mathrm{B}\left(\frac{x}{2},\frac{1}{2}\right)=
\frac{\Gamma\left(\displaystyle\frac{x}{2}\right)\Gamma\left(\displaystyle\frac{1}{2}\right)}{2~\Gamma\left(\displaystyle\frac{x+1}{2}\right)}.
\end{equation}
\end{corollary}
\begin{corollary}[Integral transforms]\label{Cor 2.5+}
For $\mathrm{Re}(x)>0$, $\Gamma(x)$ is the Mellin transform of the function $e^{-t}$, 
and $\widetilde{\Gamma}(x)$ is the Laplace transform of the function $(1-e^{-2t})^{-\frac{1}{2}}$.
\end{corollary}
\begin{proof}
From the definition of the Gamma function (\ref{1.2}), for $\textrm{Re}(x)>0$, we have
\begin{equation} 
\Gamma(x)=\int_{0}^{\infty}t^{x-1}e^{-t}dt,
\end{equation}
which implies  $\Gamma(x)$ is the Mellin transform of the function $e^{-t}$.
Then from the definition of the beta function (see \cite[Proposition 9.6.39]{Cohen}), for $\textrm{Re}(x)>0$ and $\textrm{Re}(y)>0,$ we have
$$\mathrm{B}(x,y)=\int_{0}^{\infty} e^{-xt}(1-e^{-t})^{y-1}dt,$$
and by Corollary \ref{Corollary 3.3},
$$\widetilde{\Gamma}(x)=\frac{1}{2}\mathrm{B}\left(\frac{x}{2},\frac{1}{2}\right)=\int_{0}^{\infty}e^{-xt}(1-e^{-2t})^{-\frac{1}{2}}dt,$$
which implies  $\widetilde{\Gamma}(x)$ is the Laplace transform of the function $(1-e^{-2t})^{-\frac{1}{2}}$.
\end{proof}
In the following, as an application of the integral representation (Theorem \ref{Theorem 1.10}), we will show a connection between $\widetilde{\Gamma}(x)$ and the Gauss hypergeometric function $_2F_{1}(a,b;c;z)$. 
First, recall that the Gauss hypergeometric function is defined by the series
\begin{equation}
_2F_{1}(a,b;c;z)=\sum_{n=0}^{\infty}\frac{(a)_{n}(b)_{n}}{n!(c)_{n}}z^{n}
\end{equation}
for $a,b,c \in\mathbb{C}$ with $c\not\in\mathbb{Z}_{\leq 0}$ and $|z|<1$,
where
\begin{equation}
\begin{aligned}
(\lambda)_{0}&=1,\\
(\lambda)_{n}&=\lambda(\lambda+1)\cdots(\lambda+n-1)=\frac{\Gamma(\lambda+n)}{\Gamma(\lambda)}~(n\geq 1)
\end{aligned}
\end{equation}
(see \cite[p. 136, (5) and (6)]{Wang}). 
In the following corollary, we show that $\widetilde{\Gamma}(x)$ can also be expressed in terms of $_2F_{1}(a,b;c;z)$.
\begin{corollary}\label{hyper}
For {\rm Re}$(x)>0$, we have
\begin{equation}
\widetilde{\Gamma}(x)=\frac{_2F_{1}\left(\frac{1}{2},\frac{x}{2}; \frac{x}{2}+1;1\right)}{x}.
\end{equation}
\end{corollary}
\begin{proof}
The following integral representation of the Gauss hypergeometric function $_2F_{1}(a,b;c;z)$ at $z=1$ is well-known (see, e.g., \cite[p. 156, (3)]{Wang}),
\begin{equation}
_2F_{1}(a,b;c;1)=\frac{\Gamma(c)}{\Gamma(b)\Gamma(c-b)}\int_{0}^{1}t^{b-1}(1-t)^{c-a-b-1}dt,
\end{equation}
where ${\rm Re}(c-a-b)>0$.
Thus for ${\rm Re}(x)>0$, we have
\begin{equation}
_2F_{1}\left(\frac{1}{2},\frac{x}{2};\frac{x}{2}+1;1\right)=\frac{\Gamma\left(\frac{x}{2}+1\right)}{\Gamma\left(\frac{x}{2}\right)\Gamma(1)}\int_{0}^{1}t^{\frac{x}{2}-1}(1-t)^{-\frac{1}{2}}dt.
\end{equation}
Note that 
$$\Gamma\left(\frac{x}{2}+1\right)=\frac{x}{2}\Gamma\left(\frac{x}{2}\right)$$
and $$\Gamma(1)=1,$$
we get
\begin{equation}
_2F_{1}\left(\frac{1}{2},\frac{x}{2};\frac{x}{2}+1;1\right)=\frac{x}{2}\int_{0}^{1}t^{\frac{x}{2}-1}(1-t)^{-\frac{1}{2}}dt.
\end{equation}
Then by (\ref{Th 2.3}) we have
$$_2F_{1}\left(\frac{1}{2},\frac{x}{2};\frac{x}{2}+1;1\right)=x\widetilde{\Gamma}(x),$$
which is equivalent to
\begin{equation}
\widetilde{\Gamma}(x)=\frac{_2F_{1}\left(\frac{1}{2},\frac{x}{2}; \frac{x}{2}+1;1\right)}{x}
\end{equation}
for {\rm Re}$(x)>0$.
\end{proof}
We have the following limit representation of $\widetilde{\Gamma}(x)$.
\begin{theorem}[Limit representation]\label{Theorem 1.12}
For {\rm Re}$(x)>0$, we have
\begin{equation}\label{limit}
\widetilde{\Gamma}(x)=
 \begin{cases}
        \displaystyle\lim_{n \to \infty}\displaystyle\frac{n!!}{(n-1)!!}\prod_{k=0}^{n}\left(\frac{1}{k+x}\right)^{(-1)^{k}} \quad&\textrm{if}~n~\text{is~even},\\
        \displaystyle\lim_{n \to \infty}\displaystyle\frac{(n-1)!!}{n!!}\prod_{k=0}^{n}\left(\frac{1}{k+x}\right)^{(-1)^{k}} \quad&\textrm{if}~n~\text{is~odd}.
     \end{cases}
\end{equation}
\end{theorem}
\begin{remark}
Recall the limit representation of the classical gamma function $\Gamma(x)$:
$$\Gamma(x)=\lim_{n\to\infty}\frac{n^{x}n!}{x(x+1)\cdots(x+n-1)(x+n)}$$
(see \cite[Proposition 9.6.17]{Cohen}).

According to \cite[p. 101]{Var}, the above limit representation of $\Gamma(x)$ firstly appeared  in a letter by Euler to Goldbach in 1729. Then it was independently found by Gauss, who showed it in a letter to Bessel in 1811 (see \cite[p. 33--34]{Remmert}).
\end{remark}
\begin{proof}[Proof of Theorem \ref{Theorem 1.12}]
First, we have the following limit representation of $\t\gamma_{0}$:
$$\widetilde{\gamma}_{0}(1)=\widetilde{\gamma}_{0}=\lim_{n \to \infty}\sum_{k=0}^{n}(-1)^{k}\frac{1}{k+1}$$
(see (\ref{gamma0})).
Substitute  into (\ref{1.6}), we get 
\begin{equation}\label{limitA}
\begin{aligned}
\widetilde{\Gamma}(x)&=\lim_{n \to \infty}\frac{1}{x}\exp\left(\sum_{k=1}^{n}(-1)^{k}\left(\log k-\log(k+x)\right)\right)\\
&=\lim_{n \to \infty}\exp\left(\sum_{k=1}^{n}(-1)^{k}\log k\right)\prod_{k=0}^{n}\left(\frac{1}{k+x}\right)^{(-1)^{k}}.
\end{aligned}
\end{equation}
In what follows, we shall calculate the sum $\sum_{k=1}^{n}(-1)^{k}\log k$ by cases.

If $n$ is even, that is, for $n=2m$, we have
\begin{equation*}
\begin{aligned}
\sum_{k=1}^{n}(-1)^{k}\log k&=\sum_{k=1}^{2m}(-1)^{k}\log k\\
&=-\log1+\log2-\cdots-\log(2m-1)+\log(2m)\\
&=\log\frac{(2m)!!}{(2m-1)!!}\\
&=\log\frac{n!!}{(n-1)!!}.
\end{aligned}
\end{equation*}
Then substituting to (\ref{limitA}), we get
\begin{equation}\label{limitB}
\widetilde{\Gamma}(x)
=\lim_{n \to \infty}\frac{n!!}{(n-1)!!}\prod_{k=0}^{n}\left(\frac{1}{k+x}\right)^{(-1)^{k}}.
\end{equation}
If $n$ is odd, that is, for $n=2m+1$, we have
\begin{equation*}
\begin{aligned}
\sum_{k=1}^{n}(-1)^{k}\log k&=\sum_{k=1}^{2m+1}(-1)^{k}\log k\\
&=-\log1+\log2-\cdots+\log(2m)-\log(2m+1)\\
                                  &=\log\frac{(2m)!!}{(2m+1)!!}\\
                                  &=\log\frac{(n-1)!!}{n!!}.
\end{aligned}
\end{equation*}
Then substituting to (\ref{limitA}), we get
\begin{equation}\label{limitC}
\widetilde{\Gamma}(x)
=\lim_{n \to \infty}\frac{(n-1)!!}{n!!}\prod_{k=0}^{n}\left(\frac{1}{k+x}\right)^{(-1)^{k}}.
\end{equation}
Combing (\ref{limitB}) and (\ref{limitC}), we have (\ref{limit}).
\end{proof}
\begin{corollary}\label{Corollary 3.6}
For {\rm Re}$(x)>0$, we have
$$\widetilde{\Gamma}(x)=\frac{1}{x}\prod_{k=1}^{\infty}\left(\frac{2k}{x+2k}\cdot\frac{x+2k-1}{2k-1}\right).$$
\end{corollary}
\begin{proof}
By Theorem \ref{Theorem 1.12}, if $n$ is even, that is, for $n=2m$, we have
\begin{align*}
 \widetilde{\Gamma}(x)&=\lim_{n \to \infty}\frac{n!!}{(n-1)!!}\prod_{k=0}^{n}\left(\frac{1}{k+x}\right)^{(-1)^{k}}\\
 &=\lim_{m \to \infty}\frac{(2m)!!}{(2m-1)!!}\prod_{k=0}^{2m}\left(\frac{1}{k+x}\right)^{(-1)^{k}}\\
                     &=\frac{1}{x}\lim_{m \to\infty}\left(\frac{2}{x+2}\cdot\frac{4}{x+4}\cdot\cdots\cdot\frac{2m}{x+2m}\right)\\
                     &\quad\times\left(\frac{x+1}{1}\cdot\frac{x+3}{3}\cdot\cdots
                     \cdot\frac{x+2m-1}{2m-1}\right)\\
                     &=\frac{1}{x}\lim_{m \to\infty}\prod_{k=1}^{m}\left(\frac{2k}{x+2k}\right)\left(\frac{x+2k-1}{2k-1}\right)\\
                     &=\frac{1}{x}\prod_{k=1}^{\infty}\left(\frac{2k}{x+2k}\cdot\frac{x+2k-1}{2k-1}\right).
\end{align*}
If $n$ is odd, that is for $n=2m+1$, the same reasoning leads to
$$\widetilde{\Gamma}(x)=\frac{1}{x}\prod_{k=1}^{\infty}\left(\frac{2k}{x+2k}\cdot\frac{x+2k-1}{2k-1}\right),$$
which is the desired result.
\end{proof}

\begin{remark}
	This result generalizes the classical Wallis formula, since by letting $x=1$ in Corollary \ref{Corollary 3.6}, we have
	$$
	\frac\pi2=\left(\frac{2}{1}\cdot\frac{2}{3}\right)\cdot
	\left(\frac{4}{3}\cdot\frac{4}{5}\right)\cdot
	\left(\frac{6}{5}\cdot\frac{6}{7}\right)\cdot\cdots
	=\prod_{k=1}^\infty\left(\frac{4k^2}{4k^2-1}\right).
	$$
\end{remark}

We have the following recursive formula for $\widetilde{\Gamma}(x)$.
\begin{theorem}[Recursive formula]\label{Theorem 1.14}
For {\rm Re}$(x)>0$ and $n\in\mathbb{N}$, we have 
\begin{equation}\label{1-62}
\left(\widetilde{\Gamma}(x+n)\right)^{(-1)^{n}}=\prod_{k=0}^{n-1}\left(\frac{2(x+k)}{\pi}\right)^{(-1)^{k}}\widetilde{\Gamma}(x).
\end{equation}
In particular, for $n=1$ and $2,$ we have
\begin{equation}\label{1-63}
\widetilde{\Gamma}(x+1)\widetilde{\Gamma}(x)=\frac{\pi}{2x}
\end{equation}
and
\begin{equation}\label{1-64}
\widetilde{\Gamma}(x+2)=\frac{x}{x+1}\widetilde{\Gamma}(x),
\end{equation}
respectively.
\end{theorem}
\begin{remark}
Recall the recursive formula of the classical gamma function $\Gamma(x)$:
$$\Gamma(x+1)=x\Gamma(x)$$
(see \cite[Proposition 9.6.14]{Cohen}).
\end{remark}
\begin{proof}[Proof of Theorem \ref{Theorem 1.14}]
It is known that 
$$(-1)^{n-1}\widetilde{\psi}(x+n)+\widetilde{\psi}(x)=\sum_{k=1}^{n}\frac{(-1)^{k}}{x+k-1}$$
(see \cite[(2.8)]{HK2022}), which is equivalent to
$$(-1)^{n}\widetilde{\psi}(x+n)-\widetilde{\psi}(x)=\sum_{k=0}^{n-1}\frac{(-1)^{k}}{x+k}.$$
So by (\ref{1.7}) we get
$$\frac{d}{dx}\log\frac{\widetilde{\Gamma}(x+n)^{(-1)^{n}}}{\widetilde{\Gamma}(x)}=\sum_{k=0}^{n-1}\frac{(-1)^{k}}{x+k}$$
and
\begin{align*}
\log\frac{\widetilde{\Gamma}(x+n)^{(-1)^{n}}}{\widetilde{\Gamma}(x)}
&=\int\sum_{k=0}^{n-1}\frac{(-1)^{k}}{x+k}dx\\
&=\sum_{k=0}^{n-1}(-1)^{k}\left(\log(x+k)-\log C_{k}\right)\\
&=\sum_{k=0}^{n-1}\log\left(\frac{x+k}{C_{k}}\right)^{(-1)^{k}},
\end{align*}
here $C_{k}$ are constants depending on $k$.
Thus
\begin{equation}\label{2.2}
\left(\widetilde{\Gamma}(x+n)\right)^{(-1)^{n}}=\prod_{k=0}^{n-1}\left(\frac{x+k}{C_{k}}\right)^{(-1)^{k}}\widetilde{\Gamma}(x).
\end{equation}
In particular, for $n=1$, we get
\begin{equation}\label{2.1}
\widetilde{\Gamma}(x+1)\widetilde{\Gamma}(x)=\frac{C_0}{x}.
\end{equation}

Now we determine the constant $C_0$.
By letting $x=1$ in (\ref{2.1}), we get
\begin{equation}\label{recuA}
\widetilde{\Gamma}(2)\widetilde{\Gamma}(1)=C_0.
\end{equation}
From Theorem \ref{Theorem 1.10}, we have 
\begin{equation}\label{recuB}
\t\Gamma(2)=\frac{1}{2}\textrm{B}\left(1,\frac{1}{2}\right)=1,
\end{equation}
and by Lemma \ref{Lemma 3.1}, 
\begin{equation}\label{recuC}
\t\Gamma(1)=\frac{\pi}{2}.
\end{equation}
Substituting (\ref{recuB}) and (\ref{recuC}) into (\ref{recuA}),
we determine 
$$C_{0}=\frac{\pi}{2}.$$
So (\ref{2.1}) becomes to
$$\widetilde{\Gamma}(x+1)\widetilde{\Gamma}(x)=\frac{\pi}{2x}.$$
Inductively, we have $$\left(\widetilde{\Gamma}(x+n)\right)^{(-1)^{n}}=\prod_{k=0}^{n-1}\left(\frac{2(x+k)}{\pi}\right)^{(-1)^{k}}\widetilde{\Gamma}(x),$$
which is the desired result. 
\end{proof}
From the above result, we get the special values of $\widetilde{\Gamma}(x)$ at the integers.
\begin{theorem}[Special values]\label{Corollary 1.16}
For $n\in\mathbb{N}$, we have
\begin{itemize}
\item[(1)] 
\begin{equation}\label{positive}
\widetilde{\Gamma}(n)=
 \begin{cases}
        \displaystyle\frac{(n-2)!!}{(n-1)!!}\quad &\textrm{if}~n~\text{is~even},\\
        \displaystyle\frac{(n-2)!!}{(n-1)!!}\cdot\frac{\pi}{2}\quad &\textrm{if}~n~\text{is odd}.
   \end{cases}
\end{equation}
\item[(2)]
\begin{equation}\label{negative}
\widetilde{\Gamma}(-n)=
 \begin{cases}
        \displaystyle\infty\quad &\textrm{if}~n~\text{is~even},\\
        \displaystyle 0\quad &\textrm{if}~n~\text{is odd}.
   \end{cases}
\end{equation}
\end{itemize}
\end{theorem}
\begin{remark}
Recall the special values for the classical gamma function $\Gamma(x)$:
$$\Gamma(n)=(n-1)!$$
(see \cite[Proposition 9.6.14]{Cohen}).
\end{remark}
\begin{proof}[Proof of Theorem \ref{Corollary 1.16}]
(1) We first consider the case for $n$ being odd.
If $n=3$, then by (\ref{1-64}) and Lemma \ref{Lemma 3.1}, we have
$$\widetilde{\Gamma}(3)=\frac{1}{2}\widetilde{\Gamma}(1)=\frac{(3-2)!!}{(3-1)!!}\cdot\frac{\pi}{2}.$$
Suppose that the result has been established for $n-2$, that is, 
\begin{align*}
\widetilde{\Gamma}(n-2)=&\frac{(n-4)!!}{(n-3)!!}\cdot\frac{\pi}{2}.
\end{align*}
Then by Theorem \ref{Theorem 1.12}, we have 
\begin{equation}\label{positiveA}
\begin{aligned}
\widetilde{\Gamma}(n)&=\frac{n-2}{n-1}\cdot\widetilde{\Gamma}(n-2)\\
&=\frac{n-2}{n-1}\cdot\frac{(n-4)!!}{(n-3)!!}\cdot\frac{\pi}{2}\\
&=\frac{(n-2)!!}{(n-1)!!}\cdot\frac{\pi}{2}.
\end{aligned}
\end{equation}
Now we consider even $n$.
If $n=2$, then by Theorem \ref{Theorem 1.10} we have
$$\widetilde{\Gamma}(2)=\frac{1}{2}\textrm{B}\left(1,\frac{1}{2}\right)=1=\frac{0!!}{1!!}.$$
Suppose that the result has been established for $n-2$, that is, 
$$\widetilde{\Gamma}(n-2)=\frac{(n-4)!!}{(n-3)!!}.$$
Then by Theorem \ref{Theorem 1.12}, we have
\begin{equation}\label{positiveB}
\begin{aligned}
\widetilde{\Gamma}(n)&=\frac{n-2}{n-1}\cdot\widetilde{\Gamma}(n-2)\\
&=\frac{n-2}{n-1}\cdot\frac{(n-4)!!}{(n-3)!!}\\
&=\frac{(n-2)!!}{(n-1)!!}.
\end{aligned}
\end{equation}
Combing (\ref{positiveA}) and (\ref{positiveB}), we get (\ref{positive}).

(2) Recall that the classical gamma function $\Gamma(x)$ has simple poles just at the non-positive integers (see \cite[Proposition 9.6.19]{Cohen}).
For $n$ being even, that is $n=2m~(m\in\mathbb{N})$, we have
$$\Gamma\left(-\frac{n}{2}\right)=\Gamma(-m)=\infty,$$
and by Theorem \ref{Theorem 1.18},
\begin{equation}\label{negativeA}
\widetilde{\Gamma}(-n)=\t\Gamma(-2m)=\displaystyle\frac{\sqrt{\pi}}{2}\displaystyle\frac{\Gamma\left(-\displaystyle m\right)}
{\Gamma\left(-\displaystyle m+\displaystyle\frac{1}{2}\right)}=\infty.
\end{equation}
For $n$ being odd, that is $n=2m+1~(m\in\mathbb{N})$, by Theorem \ref{Theorem 1.18} we have
\begin{equation}\label{negativeB}
\widetilde{\Gamma}(-n)=\t\Gamma(-2m-1)=\displaystyle\frac{\sqrt{\pi}}{2}\displaystyle\frac{\Gamma\left(-\displaystyle m-\frac{1}{2}\right)}
{\Gamma\left(-\displaystyle m\right)}=0.
\end{equation}
Combing (\ref{negativeA}) and (\ref{negativeB}), we get (\ref{negative}).
\end{proof}
It is known that the classical gamma function $\Gamma(x)$ is log-convex. And according to Artin \cite[p. 7]{Artin},
a function $f(x)$ defined on an interval is log-convex (or weakly log-convex) if the function $\log f(x)$ is
convex (or weakly convex). Now we show the weakly log-convexity of $\t\Gamma(x)$.
\begin{theorem}[Weakly log-convexity]\label{Theorem 3.7}
For $x>0$, $\t\Gamma(x)$ is weakly log-convex. 
\end{theorem}
\begin{proof}Since for $x>0$, we have $\Gamma(x)>0$, thus by Corollary \ref{Corollary 3.3},
$$\widetilde{\Gamma}(x)=\displaystyle\frac{\displaystyle\Gamma\left(\frac{x}{2}\right)\Gamma\left(\displaystyle\frac{1}{2}\right)}
{2~\displaystyle\Gamma\left(\frac{x+1}{2}\right)}>0.$$
Recall that (see (\ref{1.7}))
$$\t\psi(x)=\frac{d}{dx}\log\t\Gamma(x).$$
Since
\begin{equation}\label{1-53}
 \t\psi(x)=-\frac1x+\t\gamma_0+\sum_{k=1}^\infty(-1)^{k}\left(\frac1k-\frac1{k+x}\right)
 \end{equation}
 (see \cite[Theorem 3.12]{HK2022}),
 we have 
 \begin{equation}\label{convex}
 \left(\frac{d}{dx}\right)^2\log\widetilde{\Gamma}(x)=\frac{d}{dx}\t\psi(x)=\sum_{k=0}^{\infty}\frac{(-1)^{k}}{(k+x)^{2}}.
 \end{equation}
Obviously, for any $x\in(0,+\infty)$, 
$$\bigg|\frac{(-1)^{k}}{(k+x)^{2}}\bigg|=\frac{1}{(k+x)^2}<\frac{1}{k^2}, \quad k\in \mathbb N.$$
Thus by Weierstrass' test, the series
$$S(x)=\sum_{k=0}^{\infty}\frac{(-1)^{k}}{(k+x)^{2}}$$
uniformly convergent for $x\in(0,+\infty)$.
Denote $S_{n}(x)$ by the partial sum of $S(x)$, we have $$\lim_{n\to\infty}S_{n}(x)=S(x).$$
Now considering the even terms of $S_{n}(x)$, for $x>0$, we have
\begin{equation*}
\begin{aligned}
S_{2m-1}(x)&=\left(\frac{1}{x^2}-\frac{1}{(x+1)^2}\right)+\cdots+\left(\frac{1}{(x+2m-2)^2}-\frac{1}{(x+2m-1)^2}\right)\\
&> 0
\end{aligned}
\end{equation*}
and
$$S(x)=\lim_{m\to\infty}S_{2m-1}(x)\geq0.$$
Thus by (\ref{convex}), we see that
$$\left(\frac{d}{dx}\right)^2\log\widetilde{\Gamma}(x)=\sum_{k=0}^{\infty}\frac{(-1)^{k}}{(k+x)^{2}}=S(x)\geq0$$
and  $\widetilde{\Gamma}(x)$ is weakly log-convex.
\end{proof}

We have the following duplication and distribution formulas for $\widetilde{\Gamma}(x)$.

\begin{theorem}[Duplication and distribution formulas]\label{Theorem 1.18}
For {\rm Re}$(x)>0, $ we have 
\begin{equation}\label{duplication}
\widetilde{\Gamma}(2x)=\frac{\sqrt{\pi}}{2}\cdot\frac{\Gamma(x)}{\Gamma\left(x+\frac{1}{2}\right)}
\end{equation}
and more generally for a non-negative odd integer $n,$ we have the distribution formula
\begin{equation}\label{distribution}
\t\Gamma(nx)=\frac1{\sqrt n}\prod_{j=0}^{n-1} \t\Gamma\left(x+\frac jn\right)^{(-1)^{j}}.
\end{equation}
\end{theorem}
\begin{remark}Recall the duplication and the distribution formulas for classical gamma function $\Gamma(x)$, respectively:
$$\Gamma(2x)=\frac{2^{2x-1}}{\sqrt{\pi}}\cdot\Gamma(x)\Gamma\left(x+\frac{1}{2}\right),$$
$$\Gamma(nx)=n^{nx-\frac{1}{2}}(2\pi)^{\frac{1-n}{2}}\prod_{j=0}^{n-1}\Gamma\left(x+\frac jn\right)$$
(see \cite[Proposition 9.6.33]{Cohen}).
\end{remark}
\begin{proof}[Proof of Theorem \ref{Theorem 1.18}]
From Corollary \ref{Corollary 3.3} and noticing that $\Gamma\left(\frac{1}{2}\right)=\sqrt{\pi}$, we have
$$\widetilde{\Gamma}(2x)=\displaystyle\frac{\Gamma(x)\Gamma\left(\displaystyle\frac{1}{2}\right)}{2~\Gamma\left(x+ \displaystyle\frac{1}{2}\right)}=\displaystyle\frac{\sqrt{\pi}~\Gamma(x)}{2~\Gamma\left(x+\displaystyle\frac{1}{2}\right)},$$
 which is (\ref{duplication}).

Then for a non-negative odd integer $n,$ an easy calculation shows that
\begin{align}\label{zeta-dis}
	\zeta_E(s,nx)=n^{-s}\sum_{j=0}^{n-1} (-1)^j\zeta_E\left(s,x+\frac jn\right)
\end{align}
(see \cite[p. 2987, Theorem 3.9(3)]{HK2012}).
Differentiating both sides of (\ref{zeta-dis}) with respect to $s$, setting $s=0,$
and adding $\zeta_E'(0)$ on the both sides, we get
\begin{align}\label{zeta-dis2}
	\zeta_E'(0,nx)+\zeta_E'(0)&=-\zeta_E(0,x)\log n+\sum_{j=0}^{n-1} (-1)^j\left(\zeta_E'\left(0,x+\frac jn\right) +\zeta_E'(0)\right)\\
	&=-\frac12\log n+\sum_{j=0}^{n-1} (-1)^j\left(\zeta_E'\left(0,x+\frac jn\right) +\zeta_E'(0)\right)
\end{align}
by noticing that $\zeta_E(0,x)=\frac12 E_0(x)=\frac12.$ 
Then, by (\ref{Lerch-2-intro}) we have
$$\log\t\Gamma(nx)=\log\frac{1}{\sqrt{n}}+\sum_{j=0}^{n-1} (-1)^j\log\widetilde{\Gamma}\left(x+\frac jn\right),$$
which implies (\ref{distribution}).
\end{proof}

We have the following reflection equation for  $\widetilde{\Gamma}(x)$.
\begin{theorem}[Reflection equation]\label{Theorem 1.20}
For $0<\textrm{\rm Re}(x)<1,$ we have 
\begin{equation}\label{Re-2}
\frac{\widetilde{\Gamma}(x)}{\widetilde{\Gamma}(1-x)}=\cot\left(\frac{\pi x}{2}\right).
\end{equation}
\end{theorem}
\begin{remark}
Recall the reflection equation for the classical gamma function $\widetilde{\Gamma}(x)$:
\begin{equation}\label{Re-1}
\Gamma(x)\Gamma(1-x)=\frac{\pi}{\sin\pi x}
\end{equation}
(see \cite[Proposition 9.6.34]{Cohen}).
\end{remark}
\begin{proof}[Proof of Theorem \ref{Theorem 1.20}]
From Corollary \ref{Corollary 3.3} and the reflection formula of $\Gamma(x)$ (\ref{Re-1}), we have
\begin{align*}
\displaystyle\frac{\widetilde{\Gamma}(x)}{\widetilde{\Gamma}(1-x)}&=\frac{1}{2}\frac{\Gamma\left(\displaystyle\frac{x}{2}\right)
\Gamma\left(\displaystyle\frac{1}{2}\right)}{\Gamma\left(\displaystyle\frac{x+1}{2}\right)}\cdot 2\frac{\Gamma\left(\displaystyle\frac{1-x}{2}+\frac{1}{2}\right)}{\Gamma\left(\displaystyle\frac{1-x}{2}\right)\Gamma\left(\displaystyle\frac{1}{2}\right)}\\
&=\frac{\Gamma\left(\displaystyle\frac{x}{2}\right)\Gamma\left(1- \displaystyle\frac{x}{2}\right)}{\Gamma\left(\displaystyle\frac{x+1}{2}\right)\Gamma\left(1-\displaystyle\frac{x+1}{2}\right)}\\
&=\cot\left(\displaystyle\frac{\pi x}{2}\right).
\end{align*}
This is the desired assertion.
\end{proof}

\section{The properties of $\t\psi(x)$} 
In this section, we will investigate the properties of $\t\psi(x)$. First, we state its recursive formula and reflection equation.
\begin{theorem}[Recursive formula and reflection equation]\label{Theorem 1.22}
For {\rm Re}$(x)>0$ and $n\in\mathbb{N}$, we have the recursive formula
\begin{equation}\label{1.22-1}
\widetilde{\psi}(x+n)= \begin{cases}
\widetilde{\psi}(x)+\displaystyle\sum_{k=0}^{n-1}\displaystyle\frac{(-1)^{k}}{x+k}\quad &\textrm{if}~n~\text{is~even},\\
-\widetilde{\psi}(x)+\displaystyle\sum_{k=0}^{n-1}\displaystyle\frac{(-1)^{k+1}}{x+k}\quad &\textrm{if}~n~\text{is~odd},
\end{cases}
\end{equation}
and for $0<{\rm Re}(x)<1$, we have the reflection equation \cite[(2.14)]{BMM}
\begin{equation}\label{1.22-2}
\widetilde{\psi}(x)+\widetilde{\psi}(1-x)=-\frac{\pi}{\sin\pi x}.
\end{equation}
\end{theorem}
\begin{remark} Recall the recursive formula and the reflection equation for the classical digamma function $\psi(x)$:
$$\psi(x+n)=\psi(x)+\sum_{k=0}^{n-1}\frac{1}{x+k},$$
$$\psi(x)-\psi(1-x)=-\pi\cot(\pi x).$$
See \cite[Proposition 9.6.41(3)]{Cohen} and \cite[Proposition 9.6.41(5)]{Cohen}, respectively.
\end{remark}
\begin{proof}[Proof of Theorem \ref{Theorem 1.22}]
If $n$ is even, then by Theorem \ref{Theorem 1.14} we have 
\begin{equation}\label{4.1}
\widetilde{\Gamma}(x+n)=\frac{x+n-2}{x+n-1}\cdot\frac{x+n-4}{x+n-3}\cdot\cdots\cdot\frac{x}{x+1}\cdot\widetilde{\Gamma}(x)
\end{equation}
and
\begin{equation}\label{4.1}
\begin{aligned}
\log\widetilde{\Gamma}(x+n)&=\left(\log(x+n-2)+\cdots+\log x\right)\\
&\quad-\left(\log(x+n-1)+\cdots+\log(x+1)\right)+\log\widetilde{\Gamma}(x).
\end{aligned}
\end{equation}
Since $$\t\psi(x)=\frac{d}{dx}\log\t\Gamma(x)$$ (see (\ref{1.7})),
derivating  both sides of (\ref{4.1}) we get
\begin{equation}\label{4.3}
 \begin{aligned}
\widetilde{\psi}(x+n)&=\frac{1}{x+n-2}+\cdots+\frac{1}{x}-\left(\frac{1}{x+n-1}+\cdots+\frac{1}{x+1}\right)+\widetilde{\psi}(x)\\
                       &=\widetilde{\psi}(x)+\sum_{k=0}^{n-1}\frac{(-1)^{k}}{x+k}.
\end{aligned}
\end{equation}
Now we consider the odd $n$. By Theorem \ref{Theorem 1.14}, we have
$$\widetilde{\Gamma}(x+n)=\frac{x+n-2}{x+n-1}\cdot\frac{x+n-4}{x+n-3}\cdot\cdots\cdot\frac{x+1}{x+2}\cdot\widetilde{\Gamma}(x+1)$$
and
\begin{equation}\label{4.2}
\begin{aligned}
\log\widetilde{\Gamma}(x+n)&=\left(\log(x+n-2)+\cdots+\log(x+1)\right)\\
&\quad-\left(\log(x+n-1)+\cdots+\log(x+2)\right)+\log\widetilde{\Gamma}(x+1).
\end{aligned}
\end{equation}
By  (\ref{1-63}),
$$\log\widetilde{\Gamma}(x+1)+\log\widetilde{\Gamma}(x)=\log\left(\frac{\pi}{2x}\right).$$
Substitute to (\ref{4.2}) we get
\begin{align*}
\log\widetilde{\Gamma}(x+n)&=\left(\log(x+n-2)+\cdots+\log(x+1)\right)\\
&\quad-\left(\log(x+n-1)+\cdots+\log(x+2)\right)+\log\left(\frac{\pi}{2x}\right)-\log\widetilde{\Gamma}(x),
\end{align*}
and derivating both sides of the above equation, we have
\begin{equation}\label{4.4}
\begin{aligned}
\widetilde{\psi}(x+n)&=\frac{1}{x+n-2}+\cdots+\frac{1}{x+1}-\left(\frac{1}{x+n-1}+\cdots+\frac{1}{x+2}\right)-\frac{1}{x}-\widetilde{\psi}(x)\\
                        &=\sum_{k=1}^{n-1}\frac{(-1)^{k+1}}{x+k}-\frac{1}{x}-\widetilde{\psi}(x)\\
                        &=-\widetilde{\psi}(x)+\sum_{k=0}^{n-1}\frac{(-1)^{k+1}}{x+k}.
\end{aligned}
\end{equation}
Combing (\ref{4.3}) and (\ref{4.4}), we obtain (\ref{1.22-1}).
For $0<{\rm Re}(x)<1,$ by the reflection equation of $\t\Gamma(x)$ (see (\ref{Re-2})), we have 
$$\log\widetilde{\Gamma}(x)-\log\widetilde{\Gamma}(1-x)=\log\cot\left(\frac{\pi x}{2}\right).$$
Derivating both sides of  the above equation, we get
$$\widetilde{\psi}(x)+\widetilde{\psi}(1-x)=-\frac{\pi}{\sin(\pi x)},$$
which is (\ref{1.22-2}).
 \end{proof}
The following result is \cite[Corollary 3.14]{HK2022}, which gives the special values of $\t\psi(x)$ at the positive integers. We state it here for completeness.
\begin{theorem}[Special values at the positive integers]
For $n\in\mathbb{N}$, we have
$$
\widetilde{\psi}(n)=\begin{cases}
\widetilde{\gamma}_{0}+\displaystyle\sum_{k=1}^{n-1}\frac{(-1)^{k}}{k}\quad &\textrm{if}~n~\text{is even},\\
-\widetilde{\gamma}_{0}+\displaystyle\sum_{k=1}^{n-1}\frac{(-1)^{k+1}}{k}\quad &\textrm{if}~n~\text{is odd},
\end{cases}
$$
where $\t\gamma_{0}$ is the Euler constant respect to $\zeta_{E}(z,x)$ (see \eqref{1-35}).
\end{theorem}
\begin{remark}
Recall the special values for the classical digamma function $\psi(x)$ at the positive integers (see \cite[Proposition 9.6.41(3)]{Cohen}):
$$\psi(n)=-\gamma+\sum_{k=1}^{n-1}\frac{1}{k}=-\gamma+H_{n-1},$$
where $\gamma$ is the classical  Euler constant and $H_{n}$ is the harmonic series.
\end{remark}
By the recursive formula (\ref{1.22-1}), we immediately get the following result on the special values of $\t\psi(x)$ at the rational numbers.
\begin{theorem}[Special values at the rational numbers]\label{Theorem 1.26}
Let $p,q\in\mathbb{Z}$ and $q\neq 0$, $j=\frac{p}{q}$, we have 
$$
\widetilde{\psi}(n+j)=\begin{cases}
\widetilde{\psi}\left(\displaystyle\frac{p}{q}\right)+\displaystyle\sum_{k=0}^{n-1}\displaystyle\frac{(-1)^{k}q}{qk+p}\quad &\textrm{if}~n~\text{is even},\\
-\widetilde{\psi}\left(\displaystyle\frac{p}{q}\right)+\displaystyle\sum_{k=0}^{n-1}\displaystyle\frac{(-1)^{k}q}{qk+p}\quad &\textrm{if}~n~\text{is odd}.
\end{cases}
$$
\end{theorem}
\begin{remark}
Recall the corresponding result on the special values of the classical digamma function $\psi(x)$ at the rational numbers:
\begin{align*}
\psi(n+j)=\psi\left(\frac{p}{q}\right)+\sum_{k=0}^{n-1}\frac{q}{p+kq}.
\end{align*}
\end{remark}
The following is \cite[(2.5)]{HK2022}, which gives the integral representation of $\t\psi(x)$. We state it here for completeness. 
\begin{theorem}[Integral representation]\label{Theorem 1.28}
For $x>0$, we have
\begin{align*}
\widetilde{\psi}(x)&=\widetilde{\gamma}_{0}+\int_{0}^{\infty}\frac{-e^{-t}-e^{-xt}}{1+e^{-t}}dt\\
&=\widetilde{\gamma}_{0}-\int_{0}^{1}\frac{1+t^{x-1}}{1+t}dt\\
&=-\int_{0}^{\infty}\frac{e^{-xt}}{1+e^{-t}}dt,
\end{align*}
where $\t\gamma_{0}$ is the Euler constant respect to $\zeta_{E}(z,x)$ (see \eqref{1-35} and \cite[(4.1)]{BMM}).
\end{theorem}
\begin{remark}
Recall the integral representation for the classical digamma function $\psi(x)$ (see \cite[Proposition 9.6.43]{Cohen}):
\begin{align*}
\psi(x)&=-\gamma+\int_{0}^{\infty}\frac{e^{-t}-e^{-xt}}{1-e^{-t}}dt\\
&=-\gamma+\int_{0}^{1}\frac{1-t^{x-1}}{1-t}dt\\
&=\int_{0}^{\infty}\left(\frac{e^{-t}}{t}-\frac{e^{-xt}}{1-e^{-t}}\right)dt,
\end{align*}
where $\gamma$ is the classical  Euler constant.
\end{remark}

\section{Lerch-type formula}
In this section, we will show that Lerch's idea on defining of the classical gamma function $\Gamma(x)$ as a derivative of the Hurwitz 
zeta function $\zeta(z,x)$ at $z=0$ (see (\ref{Lerch-1})) is also applicable in this case. That is, we shall prove that 
\begin{equation}
\log\t\Gamma(x)=\zeta_{E}'(0,x)+\zeta_{E}'(0).
\end{equation}
First we need the following lemma.
\begin{lemma}\label{5.4}
For {\rm Re}$(x)>0$, we have
\begin{equation}\label{5.4-1}
\zeta'_{E}(0,x)=\log\Gamma\left(\frac{x}{2}\right)-\log\Gamma\left(\frac{x+1}{2}\right)-\frac{1}{2}\log2
\end{equation}
\and
\begin{equation}\label{5.4-2}
\zeta'_{E}(0)=\log\sqrt{\frac{\pi}{2}}.
\end{equation}
\end{lemma}
\begin{proof}
For Re($x)>0$, from the definition of the Hurwitz zeta function $\zeta(z,x)$ (\ref{Hu}), we have
$$\zeta\left(z,\frac{x}{2}\right)=\sum_{n=0}^{\infty}\frac{1}{\left(n+\frac{x}{2}\right)^{z}}=\sum_{n=0}^{\infty}\frac{2^{z}}{(2n+x)^{z}}$$
and
$$\zeta\left(z,\frac{x+1}{2}\right)=\sum_{n=0}^{\infty}\frac{1}{\left(n+\frac{x+1}{2}\right)^{z}}=\sum_{n=0}^{\infty}\frac{2^{z}}{(2n+1+x)^{z}}.$$
Then by comparing definition of the alternating Hurwitz zeta function $\zeta_{E}(z,x)$ (\ref{E-zeta-def}), we get
$$\zeta_{E}(z,x)=2^{-z}\left(\zeta\left(z,\frac{x}{2}\right)-\zeta\left(z,\frac{x+1}{2}\right)\right).$$
Finally, by derivating both sides of the above equation and applying Lerch's formula (\ref{Lerch-1}),
we see that 
$$\zeta'_{E}(0,x)=\log\Gamma\left(\frac{x}{2}\right)-\log\Gamma\left(\frac{x+1}{2}\right)-\frac{1}{2}\log2,$$
which is (\ref{5.4-1}).

Since for Re($z)>0$, we have $$\zeta_{E}(z)=\zeta_{E}(z,1).$$
Then from the above equation and noticing that $\Gamma(1)=1$ and $\Gamma\left(\frac{1}{2}\right)=\sqrt{\pi},$ we get
\begin{equation*}
\begin{aligned}
\zeta'_{E}(0)&=\zeta'_{E}(0,1)\\
&=\log\Gamma\left(\frac{1}{2}\right)-\log\Gamma(1)-\frac{1}{2}\log2\\
&=\log\sqrt{\frac{\pi}{2}},
\end{aligned}
\end{equation*}
which is (\ref{5.4-2}).
\end{proof}
 \begin{theorem}[Lerch-type formula]\label{Theorem 5.3}
 For {\rm Re}$(x)>0$, we have
\begin{equation}\label{Lerch-2}
\log\t\Gamma(x)=\zeta_{E}'(0,x)+\zeta_{E}'(0).
\end{equation}
 \end{theorem}
 \begin{proof}
 From Corollary \ref{Corollary 3.3}, we have
\begin{align*}
\log\widetilde{\Gamma}(x)&=\log\left(\frac{1}{2}\right)+\log\Gamma\left(\frac{1}{2}\right)+
\log\Gamma\left(\frac{x}{2}\right)-\log\Gamma\left(\frac{x+1}{2}\right)\\
                         &=-\log2+\frac{1}{2}\log\pi+\log\Gamma\left(\frac{x}{2}\right)-\log\Gamma\left(\frac{x+1}{2}\right)
\end{align*}
by noticing that $\Gamma\left(\frac{1}{2}\right)=\sqrt{\pi}.$
Then applying (\ref{5.4-1}) and (\ref{5.4-2}), we obtain
\begin{equation}\label{Lerch-3}
\begin{aligned}
\log\widetilde{\Gamma}(x)&=\zeta'_{E}(0,x)+\frac{1}{2}\log\pi-\frac{1}{2}\log2\\
                         &=\zeta'_{E}(0,x)+\log\sqrt{\frac{\pi}{2}}\\
                         &=\zeta_{E}'(0,x)+\zeta_{E}'(0),
\end{aligned}
\end{equation}
which is the desired result.
 \end{proof}

\bibliography{central}

\end{document}